\documentclass{amsart}
\usepackage{amsfonts}
\usepackage{amsmath}
\usepackage{amssymb}
\usepackage{amscd}

\usepackage[all,cmtip]{xy}

\newtheorem{thm}{Theorem}[section]
\newtheorem{lm}[thm]{Lemma}
\newtheorem{cor}[thm]{Corollary}
\newtheorem{prop}[thm]{Proposition}
\newtheorem{exm}[thm]{Example}
\newtheorem{prob}[thm]{Problem}

\theoremstyle{definition}
\newtheorem*{definition}{Definition}

\newcommand{\gen}{\text{\rm gen}}

\newcommand{\Z}{\mathbb{Z}}
\newcommand{\Q}{\mathbb{Q}}
\newcommand{\R}{\mathbb{R}}

\DeclareMathOperator{\Sing}{{Sing}}

\DeclareMathOperator{\Soc}{\rm Soc}


\begin{document}

\title[Hereditarily and Super Bassian Modules ...]{Hereditarily and Super Bassian Modules \\ over Certain Rings}

\author[P. V. Danchev]{Peter Vassilev Danchev}
\address{Bulgarian Academy of Sciences, Institute of Mathematics and Informatics, Sofia 1113, Bulgaria}
\email{danchev@math.bas.bg; pvdanchev@yahoo.com}
\author[T. C. Quynh]{Truong Cong Quynh}
\address{Faculty of  Mathematics, The University of Danang - University of Science and Education,  Da Nang 550000, Vietnam}
\email{tcquynh@ued.udn.vn}
\author[J. \v{Z}emli\v{c}ka]{Jan \v{Z}emli\v{c}ka}
\address{Department of Algebra, MFF UK, Sokolovsk\' a 83, 186 75 Praha 8, Czechia}
\email{zemlicka@karlin.mff.cuni.cz; jan.zemlicka@matfyz.cuni.cz}

\begin{abstract} We characterize in certain basic cases when a module over a ring is either {\it hereditarily Bassian} or {\it super Bassian} in the sense that either each its proper submodule is Bassian or, respectively, each its proper epimorphic image is Bassian. We prove several structural criteria for both hereditarily Bassian and super Bassian modules over non-primitive Dedekind prime rings, and in particular Dedekind domains. Over these rings, we establish that a singular module is super Bassian exactly when it is Bassian, which is true if and only if it is Bassian. In addition, for an arbitrary (not necessarily singular) module over a non-primitive Dedekind prime ring, the property of being super Bassian curiously implies the property of being hereditary Bassian always.

Our results somewhat continue and supply recent results due to Tuganbaev in Mathematics (2026) and Blacher in J. Algebra (2026).
\end{abstract}

\date{\today}

\maketitle

\section*{Introduction and Motivation}

Let $M$ be a module over a ring $R$.
Imitating \cite{DKTZ} (see too \cite{CDG} and \cite{DG} as original sources), we say that $M$ is {\it Bassian}, provided that the existence of a monomorphism $M \to M/N$ for some submodule $N$ of $M$ forces that $N$ is the trivial submodule of $M$, that is, $N=(0)$. More generally, if $N$ is a (possibly proper) direct summand of $M$, then $M$ is said to be {\it generalized Bassian}.

Moreover, if $L$ is a module such that each its factor-module $L/K$ is Bassian for any (proper) submodule $K\leq L$, then $L$ is called {\it super Bassian}. It is quite surprisingly that super Bassian modules are just Hopfian (see, e.g., \cite[Proposition 3.6 (3)]{DKTZ}). On the other hand, if $T$ is a module such that each its (proper) submodule is Bassian, then $T$ is called {\it hereditarily Bassian}.

As a pioneer work on the present subject, Rowen and Small explored in \cite{RS} when some large classes of (affine) algebras are Bassian proving that any semi-prime affine PI-algebra over a field is Bassian, and any semi-prime affine Goldie algebra of finite GK dimension satisfying the ACC property on semi-prime ideals is Bassian.

Likewise, recently, in \cite{DG} and \cite{K} were studied and characterized those Bassian Abelian groups whose proper subgroups are again Bassian, as well as those Bassian Abelian groups whose proper homomorphic images are Bassian too.

On the other hand, very recently, Blancher examined in \cite{B} those rings whose all proper subrings are either Noetherian or Artinian.

That is why, it is quite logical to continue the investigations in this way for other algebraic structures like modules over some classical types of rings. We, however, treat here the more global case of modules over a few large classes of rings such as (possibly non-commutative) non-primitive Dedekind prime rings and principal ideal domains. Our work is organized as follows: We distribute the main results into two subsequent sections which are devoted, respectively, to the super Bassian and hereditary Bassian cases. In the first section, we establish that the property of any module over a non-primitive Dedekind prime ring to be super Bassian is equivalent to the existence of a finitely generated submodule of the former module such that the corresponding quotient is singular and all its primary components are again finitely generated (see Theorem~\ref{super Bassian}). In the second one, we establish that the property of any module over a non-primitive Dedekind prime ring to be hereditary Bassian is tantamount to the fact that each primary component of its singular submodule is finitely generated and the existence of a finitely generated submodule that is singular modulo the given module. Moreover, each module over a semi-local ring with nilpotent Jacobson radical is hereditarily Bassian precisely when its socle is finitely generated (see Theorems~\ref{hereditarily Bassian} and \ref{artinian}, respectively).

Traditionally, $E(M)$ denotes an injective envelope of $M$. If $R$ is a domain, then $\tau(M)$ stands for the torsion part of $M$. Our basic notation and terminology are mostly standard, whereas the new notions will be specified in detail in what follows.

\section{Super Bassian Modules}

We begin here by listing once again our pivotal instrument.

\begin{definition} Let $R$ be a ring. A right $R$-module $M$ is said to be {\it super Bassian} if every epimorphic image of $M$ is Bassian.
\end{definition}

The next statement is worthy of mentioning.

\begin{exm}
Any right duo ring $R$ is super Bassian as a right $R$-module.
\end{exm}

\begin{proof}
Suppose $R$ is a right duo ring, and both $I$ and $H$ are right ideals of $R$ with $I\leq H$. Letting $f: R/I \to R/H$ be a monomorphism of right $R$-modules, we may assume that $f(1+I) = a + H$ for some $a \in R$. For all $x\in H$, since $R$ is right duo, we have $ax\in H$ and so $f(x+I)=ax+H=0$. We thus deduce that $f$ is a monomorphism and hence $x\in I$. It follows now that $I=H$. Therefore, this shows that $R_R$ is super Bassian, as stated.
\end{proof}

The following technical observation is pretty obvious, so that we can omit its proof.

\begin{lm}
Let $M$ be a right $R$-module. Then, the following conditions are equivalent:
\begin{enumerate}
\item $M$ is super Bassian.
\item each quotient of $M$ is super Bassian.
\item if the sequence $0\to M/N\to M/K$ with $N\leq K\leq M$ is exact, then $N=K$.
\end{enumerate}
\end{lm}

Since direct summands are isomorphic to factors of a module, we obtain in addition as an consequence that the class of super Bassian modules is closed under taking direct summands.

A module \( M \) is called \emph{distributive}, provided $A \cap (B + C) = (A \cap B) + (A \cap C)$ holds for all submodules \( A, B, \) and \( C \) of \( M \). It is well known that if $M$ is distributive, then \( M/N \) has square-free socle for each submodule \( N \subseteq M \).

\begin{prop}
Let $M$ be a distributive right $R$-module. Then, the following two conditions are equivalent:
\begin{enumerate}
\item $M$ is super Bassian.
\item Every factor module of $M$ is generalized Bassian.
\end{enumerate}
\end{prop}

\begin{proof}
(1) $\Rightarrow$ (2). It is obvious, because every Bassian module is generalized Bassian.	
	
(2) $\Rightarrow$ (1). Assume that every factor module of $M$ is generalized Bassian. Furthermore, let $K$ and $I$ be submodules of $M$ with $I\leq K$ and $\alpha: M/I\to M/K\cong (M/I)/(K/I)$ be a monomorphism. Since $M/I$ is generalized Bassian, $K/I$ is semi-simple and $(K/I)^{(\omega)}$ is embeddable into $M/I$ by
\cite[Proposition 4.4]{DKTZ}. But, as noticed above, $M/I$ has square-free socle, we infer that $K/I=(\bar{0})$ whence $K=I$. We, thereby, conclude that $M$ is super Bassian, as wanted.
\end{proof}

Since an infinite direct power of a nonzero module is not Bassian, we can formulate the following easy observation.

\begin{lm}\label{super} A module $M$ over a ring $R$ is not super Bassian if there exists $N\le M$ and a non-zero module $S$ such that $M/N\cong S^{(\omega)}$.
\end{lm}

The next claim is also helpful for further applicable purposes and so worthy of documenting.

\begin{lm}\label{perfect} Let $M$ be an infinitely generated module over a right perfect ring. Then,
\begin{enumerate}
\item there exists a simple module $S$ and a submodule $N$ such that $M/N\cong S^{(\omega)}$;
\item M is not super Bassian.
\end{enumerate}
\end{lm}

\begin{proof} (1) Let $R$ be a right perfect ring with Jacobson radical $J$, and let $M$ be an infinitely generated module. Since $MJ$ is superfluous in $M$, the factor $M/(MJ)$ has the structure of an infinitely generated module over the semi-simple ring $R/J$, thus it is an infinitely generated semi-simple module. Notice that there exists up to an isomorphism only finitely many simple modules over each perfect ring, hence there exists a simple module $S$ and a submodule $N$ such that $MJ\le N\le M$ and
$M/N\cong S^{(\omega)}$.

(2) $M$ is not super Bassian according to (1) and Lemma~\ref{super}.
\end{proof}

Recall that a module $M$ is called {\it q.f.d.} if each quotient of $M$ has finite Goldie dimension.
As a consequence of the previous assertion, we derive:

\begin{prop} Let $M$ be a right module over a right Artinian ring $R$. Then, the following three conditions are equivalent:
\begin{enumerate}
\item $M$ is super Bassian.
\item $M$ is finitely generated.
\item $M$ is a q.f.d. module.
\end{enumerate}	
\end{prop}

\begin{proof} (1) $\Rightarrow$ (2). It follows from Lemma~\ref{perfect}(2)

(2) $\Rightarrow$ (1) and (3). Since a finitely generated module over an Artinian ring is necessarily of finite length, it is a super Bassian q.f.d. module.
	
(3) $\Rightarrow$ (2). Assume that $M$ is an infinitely generated module. Then by
Lemma~\ref{perfect}(1), there exist a simple module \( S \) and a submodule \( N \) of $M$ such that \(M/N \cong S^{(\omega)}\), which is a contradiction to the finite Goldie dimension of $M/N$, as expected.
\end{proof}

We now intend to look at the more complicated situation of when a module over a right Noetherian ring is super Bassian by establishing the following (see \cite[Example 2.1]{DKTZ}).

\begin{lm} Any finitely generated module over a right Noetherian ring is super Bassian.
\end{lm}

\begin{proof} Recall that Noetherian module is Bassian. Since finitely generated module over a right Noetherian ring is Noetherian, it is super Bassian, as asked.
\end{proof}

We are now prepared to prove the following assertion.

\begin{lm} Let $\mathcal M$ be the set of all maximal ideals of a non semi-simple abelian regular ring $R$, and let $M$ be an $R$-module. If $M$ is super Bassian, then $$\gen(M)\le |\mathcal M|\le 2^{|R|}.$$
\end{lm}

\begin{proof} Observe first that $\mathcal M$ is an infinite set, because $R$ is not semi-simple. Given $I\in \mathcal M$, then $M/MI$ is finitely generated invoking Lemma~\ref{super}, so we may fix a finitely generated submodule $F_I$ such that $M=F_I+MI$. Therefore, one writes that $M=\sum_{I\in \mathcal{M}}~ F_I$ since $J(M)=(0)$. Thus, it automatically follows that $$\gen(M)\le |\omega\times\mathcal M|=|\mathcal M|\le 2^{|R|},$$ as requested.
\end{proof}

The following statement could be useful.

\begin{prop} If $M$ is a module which is not Bassian, then it contains a countable generated submodule which is not Hopfian.
\end{prop}

\begin{proof} Let $M$ be non-Bassian and let $N\le M$ be a non-zero submodule such that $\iota:M\to M/N$ is an embedding. Let $0\ne x\in N$ and define a sequence of elements $x_n\in M$ inductively by the rule $x_0=x$ and $x_{n+1}+N=\iota(x_n)$. Put $X:=\sum_{n\ge 0}x_nR$ and note that
\[
\iota(X)=\sum_{n\ge 0}\iota(x_n)R =\sum_{n\ge 1} (x_{n}R+N)/N = \sum_{n\ge 0} (x_{n}R+N)/N = (X+N)/N,
\]
because $x_0\in N$. Thus, $\iota$ injectively maps $X$ onto $(X+N)/N\cong X/(X\cap N)$ and $x_0\in X\cap N\ne (0)$. We thereby have proved that $X$ is countably generated non-Hopfian submodule.
\end{proof}

In the rest of the section, we describe super Bassian modules over non-commutative generalization of Dedekind domains. To that end, we need to recall some useful notions from \cite{T}.

Let $M$ be a module. We denote by $\Sing(M)$ the {\it singular submodule} of $M$ and we say that $M$ is {\it non-singular} if $\Sing(M)=(0)$, it is {\it singular} if $\Sing(M)=M$, and it is {\it Goldie-radical} if $M/\Sing(M)=\Sing(M/\Sing(M))$. The abbreviation $B\trianglelefteq M$ means that $B$ is an {\it essential submodule} of $M$ and a {\it HNP} ring means a {\it hereditary Noetherian prime} ring. Likewise, a {\it non-primitive Dedekind prime} ring is HNP with the classical ring of fractions $Q$ such that it is {\it not} primitive and every non-zero ideal is invertible in $Q$. For some basic properties of these notions and other non-explained terminology, we refer to \cite{T}.

\medskip

We now formulate and prove the following useful technicality.

\begin{lm}\label{sing} Let $R$ be a right non-singular ring, and $A$, $B$ submodules of a module $M$. Then,
\begin{enumerate}
\item[(1)] if $B\trianglelefteq M$, then $M/B$ is singular.
\item[(2)] if $A\trianglelefteq B$ and $M/B$ is singular, then $M/A$ is singular.
\item[(3)] the following is equivalent:
\begin{enumerate}
\item[(a)] $M/B$ is singular.
\item[(b)] $B+\Sing(M)\trianglelefteq M$.
\item[(c)] there exists a submodule $H$ of $B$ such that $H\cap\Sing(M)=(0)$ and $M/H$ is singular.
 \end{enumerate}
 \end{enumerate}
\end{lm}

\begin{proof} (1) It is pretty clear, because $xR\cap B\trianglelefteq xR$ for each non-zero $x\in M$.

(2) Since both $B/A$ and $M/B$ are singular, the quotient $M/A$ is a Goldie-radical module. As $R$ is right non-singular, the assertion follows at once from \cite[Remark 11.3]{T}.

(3) Set $S:=\Sing(M)$.

(a) $\Rightarrow$ (b). Proving indirectly, we assume that $B+S$ is not essential in $M$, i.e., there exists non-zero $m\in M$ such that $mR\cap(B+S)=(0)$. Thus, $mR\cap B=(0)$, hence $mR$ embeds into $M/B$. As $mR$ is non-singular, $M$ is not a singular module.

(b) $\Rightarrow$ (c). Let $H$ be a maximal submodule of $B$ with respect to the condition $H\cap S=(0)$. Thus, $H+S\trianglelefteq B+S \trianglelefteq M$ and so $M/(H+S)$ is singular by (1). Furthermore, $(S+H)/H$ is singular, which gives that $M/H$ is Goldie-radical. Therefore, $M/H$ is singular again applying \cite[Remark 11.3]{T}.

(c) $\Rightarrow$ (a). It is straightforward.
\end{proof}

We are now prepared to establish the following necessary condition for a super Bassian module.

\begin{prop}\label{HNP} Let $R$ be a non-primitive HNP ring, and let $F$ be a finitely generated submodule of a module $M$. If $M/F$ is singular such that the $P$-primary component of $M/F$ is Noetherian for each maximal invertible ideal $P$, then $M$ is super Bassian.
\end{prop}

\begin{proof} First, note that $R$ is non-singular, because it is obviously hereditary. We will show, for arbitrary factor $\tilde{M}$ of $M$, that it satisfies the hypothesis of \cite[Theorem 3]{T}, i.e., that $\tilde{M}/\Sing(\tilde{M})$ is finite dimensional and every primary component of $\Sing(\tilde{M})$ is Noetherian, which will suffice to get the desired assertion.

To that target, let $X$ be an arbitrary submodule of $M$, and denote $\tilde{M}=M/X$, $\tilde{F}=F+X/X$ and
$\tilde{S}=\Sing(\tilde{M})$.

Since $M/F$ is singular, observe that $\tilde{M}/(\tilde{F}+\tilde{S})$ is singular as well. Furthermore, \cite[Proposition 1.1]{T} reaches us that $\tilde{M}/\tilde{S}$ is non-singular, whence
$(\tilde{F}+\tilde{S})/\tilde{S} \trianglelefteq\tilde{M}/\tilde{S}$ employing Lemma~\ref{sing}. As $F$ is finitely generated and $R$ is Noetherian, one knows that $(\tilde{F}+\tilde{S})/\tilde{S}$ is Noetherian too, which means that $\tilde{M}/\tilde{S}$ has finite Goldie dimension.

Next, suppose that $P$ is a maximal invertible ideal of $R$. Thanks to Lemma~\ref{sing}, there exists a submodule $\tilde{H}$ of $\tilde{F}$ such that $\tilde{H}\cap \tilde{S}=(0)$ and $\tilde{M}/\tilde{H}$ is singular. Then, $\tilde{S}$ and so the $P$-primary component $\tilde{S}(P)$ of $\tilde{S}$ embeds into $\tilde{M}/\tilde{H}$. Moreover, $(\tilde{S}(P)+\tilde{F})/\tilde{F}$ embeds into the $P$-primary component of $\tilde{M}/\tilde{F}$, which is a quotient of the $P$-primary component of $M/F$ viewing \cite[Remark 8]{T}. Hence, $$(\tilde{S}(P)+\tilde{F})/\tilde{F}\cong \tilde{S}(P)/(\tilde{S}(P)\cap \tilde{F})$$ is Noetherian by hypothesis. Since $\tilde{S}(P)\cap \tilde{F}$ is clearly Noetherian, we have proved that $\tilde{S}(P)$ is Noetherian as well.

Consequently, it remains to apply \cite[Theorem 3]{T}, which claims that $\tilde{M}$ is Bassian. This proves that $M$ is super Bassian after all, as wanted.
\end{proof}

Before to proceed, we come to the following.

\begin{lm}\label{f.g.submodule} If $M$ is a super Bassian module over a non-primitive Dedekind prime ring, then there exists a finitely generated non-singular submodule $H$ of a module $M$ such that $M/H$ is singular.
\end{lm}

\begin{proof} First, observe that a Dedekind prime ring is HNP and so a non-singular Goldie prime ring (see, e.g., \cite[Remark 2]{T}), which means that we may apply Lemma~\ref{sing} and \cite[Lemma 4]{T}.

Put $S:=\Sing(M)$ and notice that, by Lemma~\ref{sing}, there exists a submodule $B$ of $M$ such that $B\cap S=(0)$, $M/B$ is singular, and $B$ is non-singular. Assume to contrary that $B$ has infinite Goldie dimension. Then, looking at \cite[Lemma 4.3]{T}, there exists a free submodule $F$ of $B$ of infinite rank.

Now, denote by $G$ a maximal submodule of $B$ satisfying the condition $G\cap F=(0)$, and suppose that $P$ is an arbitrary maximal invertible ideal of $R$. Since
\[
G\oplus FP\oplus S\ \ \trianglelefteq\ \  G\oplus F\oplus S\ \  \trianglelefteq\ \  B\oplus S\ \  \trianglelefteq\ \  M,
\]
we obtain that $\tilde{M}=M/(G\oplus FP\oplus S)$ is singular again with the aid of Lemma~\ref{sing}. Note that $\tilde{M}$ is Bassian and $F/FP$ is embedding into $P$-primary component of $\tilde{M}$. But, as $F/FP$ is infinitely generated, we get a contradiction with \cite[Theorem 1]{T}, which asserts that all primary components of $\tilde{M}$ are Noetherian.

Thus, $B$ is of finite dimension meaning that $B$ contains a finitely generated essential submodule $H$. As $H+S\trianglelefteq B+S\trianglelefteq M$, the module $M/H$ is singular with the help of Lemma~\ref{sing}, and we are done.
\end{proof}

We now have at our disposal all the ingredients necessary to attack the following statement.

\begin{thm}\label{super Bassian}
The following conditions are equivalent for a module $M$ over a non-primitive Dedekind prime ring:
\begin{enumerate}
\item[(1)] $M$ is super Bassian.
\item[(2)] There exists a finitely generated non-singular submodule $H$ of $M$ such that $M/H$ is singular and all primary components of $M/H$ are Noetherian.
\item[(3)] There exists a finitely generated submodule $H$ of $M$ such that $M/H$ is singular and all primary components of $M/H$ are finitely generated.
\end{enumerate}
\end{thm}

\begin{proof} (1) $\Rightarrow$ (2). Consulting with Lemma~\ref{f.g.submodule}, there exists a finitely generated non-singular submodule $H$ of $M$ such that $M/H$ is singular. As $M/H$ is Bassian, all primary components of $M/H$ are Noetherian thanking to \cite[Theorem 1]{T}.

(2) $\Rightarrow$ (3). It is just obvious.

(3) $\Rightarrow$ (1). Since $R$ is right Noetherian owing to \cite[Remark 2]{T}, all primary components of $M/H$ are Noetherian. Now, $M$ is super Bassian in conjunction with Proposition~\ref{HNP}.
\end{proof}

The following immediate consequences of \cite[Theorem 1]{T} and the previous theorem clarify the connection between Bassian and super Bassian over non-primitive Dedekind prime rings.

\begin{cor}\label{sB} Let $M$ be a module over a non-primitive Dedekind prime ring.
\begin{enumerate}
\item[(1)] If $M$ is singular, then the following conditions are equivalent:
\begin{enumerate}
\item[(a)] $M$ is super Bassian.
\item[(b)] Each primary component of $M$ is Noetherian.
\item[(c)] $M$ is Bassian.
\end{enumerate}
\item[(2)] $M$ is super Bassian if, and only if, $M$ is Bassian and there exists a finitely generated submodule $H$ of $M$ such that $M/H$ is singular Bassian.
\end{enumerate}
\end{cor}

We conclude the section by a reformulation of the result in commutative case equipped with one easy example.

\begin{cor}\label{sBcom} A module $M$ over a Dedekind domain is super Bassian if, and only if, there exists a finitely generated torsion-free submodule $F$ of $M$ such that $M/F$ is torsion and all primary components of $M/F$ are finitely generated.
\end{cor}

Note that the previous corollary follows in case of Abelian groups also from \cite[Theorem 2.5]{DG}.

\begin{exm}
If $M$ is a torsion-free module of rank 1 over a principal ideal domain $R$ such that $pM\ne M$ for each prime element $p\in R$, then $M$ is super Bassian having in mind Corollary~\ref{sBcom}. In particular, the Abelian group
\[
A=\{\frac{z}{\prod p_i}\in \Q \mid z\in \Z,\  p_1,\dots,p_n ~ {\rm are ~ distinct ~ primes}\}
\] is always super Bassian.
\end{exm}

\section{Hereditarily Bassian Modules}

We start here by recalling once again the second version of the strengthened notion of a Bassian module.

\begin{definition} Let $R$ be a ring. A right $R$-module $M$ is said to be a {\it hereditarily Bassian} module if every submodule is Bassian.
\end{definition}

Observe immediately that the class of all hereditarily Bassian modules is closed under submodules.

\medskip

Our first preliminary observations concern the necessary condition of being a hereditarily Bassian module over a semi-perfect ring.

\begin{lm}\label{semi-perfect} If $M$ is a hereditarily Bassian module over a semi-perfect ring, then the socle of $M$ is finitely generated.
\end{lm}

\begin{proof} Showing indirectly, assume that $\Soc(M)$ is infinitely generated. Since there exists only finitely many non-isomorphic simple modules, $\Soc(M)$ contains a submodule of a module $S^{(\omega)}$ for a suitable simple module $S$, which is not Bassian, and thus $M$ is not hereditarily Bassian.
\end{proof}

\begin{lm}\label{noetherian2} Any finitely generated module over a right Noetherian ring is hereditarily Bassian.
\end{lm}

\begin{proof} It is apparent, since every Noetherian module is Bassian, every submodule of a Noetherian module is Noetherian as well, and every finitely generated module over a right Noetherian ring is Noetherian too.
\end{proof}

Now, we are ready to characterize hereditarily Bassian modules over quasi-Frobenius rings.

\begin{prop}\label{qF} The following three conditions are equivalent for a module $M$ over a quasi-Frobenius ring.
\begin{enumerate}
\item $M$ is hereditarily Bassian.
\item $M$ is finitely generated.
\item $\operatorname{Soc}(M)$ of $M$ is finitely generated.
\end{enumerate}	
\end{prop}

\begin{proof}
(1) $\Rightarrow$ (3). As a quasi-Frobenius ring is semi-perfect, every hereditarily Bassian module has finitely generated socle by Lemma~\ref{semi-perfect}.

(3) $\Rightarrow$ (2). Since a quasi-Frobenius ring is (right) Artinian, the module $M$ is semi-Artinian, hence $\Soc(M)$ is essential in $M$. Furthermore, the injective envelope $E(\Soc(M))\cong E(M)$ of finitely generated $\Soc(M)$ is a finitely generated projective module, which proves that $E(M)$ as well as $M$ are Noetherian.

(2) $\Rightarrow$ (1). Since a quasi-Frobenius ring is (right) Noetherian, the assertion follows from Lemma~\ref{noetherian2}.
\end{proof}

A natural question which immediately arises is of whether or not the previous claim remains true for right Artinian rings. Before we answer it, we need to prove two  assertions. The first one is an elementary but helpful technical observation.

\begin{lm}\label{step} Let $Q,S,X,Y$ be submodules of a module $A$ such that $Q\subseteq S\cap X$ and $X\subseteq Y$. If $S/Q=\Soc(A/Q)$ and $\Soc(A/X)\subseteq Y/X$, then $S\subseteq Y$.
\end{lm}

\begin{proof} Since $(S+X)/X\cong ((S/Q)+(X/Q)/(X/Q)$ is a semi-simple module, we get
\[(S+X)/X\subseteq \Soc(A/X)\subseteq Y/X,
\]
which immediately implies $S\subseteq S+X\subseteq  Y$.
\end{proof}

\begin{lm}\label{semiart} Let $M$ be a semi-Artinian module over a semi-local ring $R$ with the Jacobson radical $J$ such that $\Soc(M)$ is finitely generated. If there exists natural $n$ satisfying $MJ^n=(0)$, then $M$ is Bassian.
\end{lm}

\begin{proof} Proving indirectly the stated claim, we assume that $M$ is not Bassian. Then, there exists $X\ne (0)$ and an embedding $\nu: M\to M/X$, i.e., $M\cong\nu(M)$. Set $X_0:=0$, $X_1:=X$, $M_0:=M$, and let $M_1$ be a submodule of $M$ satisfying $X_1\subseteq M_1$ and $M_1/X_1=\nu(M)$. Since $$M\cong M_0/X_0\cong M_1/X_1,$$ we can proceed by induction to obtain sequences of submodules $(X_n\mid n<\omega)$ and $(M_n\mid n<\omega)$ such that $X_{i}\subsetneq X_{i+1}$, $M_i \supseteq M_{i+1}$ and $M_i/X_i\cong M$ for each $i<\omega$. Put $A:=\bigcup_{n<\omega} X_n$. Since $A$ is semi-Artinian, there exists a socle sequence $(S_\alpha\mid S_\alpha<\kappa)$ of $A$, i.e., a smooth sequence of submodules such that $$S_0=0, ~ S_{\alpha+1}/S_\alpha=\Soc(A/S_\alpha) ~ {\rm and} ~ A=\bigcup_{\alpha<\kappa} S_\alpha.$$

We will show now that $\kappa$ is infinite. First, we will prove by induction on $\alpha$ that, for each $\alpha<\min(\omega,\kappa)$, there exists an index $i_\alpha<\omega$ for which $S_\alpha\subseteq X_{i_\alpha}$. As $S_0=X_0=(0)$, we may put $i_0=0$. Let $S_\alpha\subseteq X_{i_\alpha}$. Since $\Soc(A/X_{i_\alpha})\subseteq \Soc(M_{i_\alpha}/X_{i_\alpha})$ is finitely generated and $$A/X_{i_\alpha}=\bigcup_{n>i_\alpha} X_n/X_{i_\alpha},$$ there exists an index $k>i_\alpha$ such that $\Soc(X_k/X_{i_\alpha})=\Soc(A/X_{i_\alpha})$. However, as $S_{\alpha+1}/S_\alpha=\Soc(A/S_\alpha)$, we may apply Lemma~\ref{step} claiming that $S_{\alpha+1}\subseteq X_k$. We next put $i_{\alpha+1}:=k$, thus finishing the induction.

Now, if we assume that $\kappa$ is finite, we get that
$$
A=S_{\kappa-1}\subseteq X_{i_{\kappa-1}}\subsetneq A,
$$
an obvious contradiction.

Thus, $A$ and so $M$ have infinite socle lengths, which means that $MJ^n\ne (0)$ for each $n<\omega$.
\end{proof}

Now, we are able to generalize Proposition~\ref{qF} in the following way.

\begin{thm}\label{artinian} Let $R$ be a semi-local ring with nilpotent Jacobson radical and let $M$ be a module. Then, $M$ is hereditarily Bassian if, and only if, $\Soc(M)$ is finitely generated.
\end{thm}

\begin{proof} $(\Rightarrow)$. It follows directly from Lemma~\ref{semi-perfect}.
	
$(\Leftarrow)$. If $\Soc(M)$ is finitely generated, then the socle of every submodule $N$ of $M$ is finitely generated. Furthermore, there exists $n$ such that $J^n=(0)$ for the Jacobson radical of $R$, whence $NJ^n=N0=(0)$. Now, it remains to employ Lemma~\ref{semiart} getting the desired conclusion.
\end{proof}

Note that the hypothesis of the previous theorem is satisfied by every one-sided Artinian ring. The following example unambiguously shows that the condition (2) of Proposition~\ref{qF} is {\it not} an equivalent condition of hereditarily Bassian modules over right Artinian rings.

\begin{exm} Let $$R=\begin{pmatrix} \Q&\R\\0&\R\end{pmatrix} ~ \text{and} ~ M=\begin{pmatrix} \R&\R\\ \R&\R\end{pmatrix}.$$ Then, a simple inspection demonstrates that $R$ is a right Artinian ring (which is not left Artinian), $M_R$ has a natural structure of a right $R$-module, and $$\Soc(M_R)=\begin{pmatrix} 0&\R\\ 0&\R\end{pmatrix}$$ is a finitely generated $R$-module. Thus, $M_R$ is an infinitely generated module which is hereditarily Bassian viewing Theorem~\ref{artinian}.

So, there are right Artinian rings $R$ over which the injective envelope $E(R_R)=E(\Soc(R_R))$ is definitely not finitely generated.
\end{exm}

Similarly, as in the case of super Bassian modules, we now characterize hereditarily Bassian modules over non-primitive Dedekind prime rings.

\begin{lm}\label{hereditarily} Let $F$ and $N$ be submodules of a module $M$ over a right hereditary right Noetherian ring $R$. If $F$ is finitely generated and $M/F$ is singular, then there exists a finitely generated submodule $G$ of $N$ such that $N/G$ is singular.
\end{lm}

\begin{proof} Firstly, denote by $S$ the singular submodule of $M$ and observe that $\Sing(N)=S\cap N$. Since $M/F$ is singular, $F+S\trianglelefteq M$ by Lemma~\ref{sing}(3). As $M/S$ is non-singular and $M/(F+S)$ is singular, one checks that $(F+S)/S\trianglelefteq M/S$ again by Lemma~\ref{sing}(3). Thus,
\[
((N+S)/S)\cap ((F+S)/S) \trianglelefteq (N+S)/S.
\]
Clearly, $(F+S)/S$ is Noetherian, hence the module $((N+S)/S)\cap ((F+S)/S)$ is finitely generated. This guarantees existence of a finitely generated submodule $G$ of $N$ such that $$G+S=(N+S)\cap (F+S),$$ which is an essential submodule of the module $N+S$. Finally, $$G+(S\cap N)=(G+S)\cap N\trianglelefteq N$$ via the modular law, and $$N/(G+(S\cap N))= N/(G+\Sing(N))$$ is singular. Therefore, Lemma~\ref{sing} informs us that $N/G$ is singular as well, finishing the argumentation.
\end{proof}

We now have at our disposal all the machinery needed to prove the following equivalencies.

\begin{thm}\label{hereditarily Bassian}
The following items are equivalent for a module $M$ over a non-primitive Dedekind prime ring:
\begin{enumerate}
\item[(1)] $M$ is hereditarily Bassian.
\item[(2)] Each primary component of $\Sing(M)$ is Noetherian and there exists a finitely generated non-singular submodule $G$ of $M$ such that $M/G$ is singular.
\item[(3)] Each primary component of $\Sing(M)$ is finitely generated and there exists a finitely generated submodule $G$ of $M$ such that $M/G$ is singular.
\end{enumerate}
\end{thm}

\begin{proof} Put $S:=\Sing(M)$.

(1) $\Rightarrow$ (2). Since $S$ is Bassian, each primary component of $S$ is Noetherian in accordance with \cite[Theorem 1]{T}. However, from Lemma~\ref{sing} follows the existence of a submodule $H$ of $M$ such that $H\cap S=(0)$ and $M/H$ is singular. Evidently, $H$ is non-singular. Using the same argument as in the proof of Lemma~\ref{f.g.submodule} and assuming $H$ has infinite dimension, we obtain a free submodule $F$ of $H$ which must have infinite rank bearing in mind \cite[Lemma 4.3]{T}. This, in turn, contradicts the hypothesis that all submodules of $M$ are Bassian. We thus have shown that $H$ has infinite dimension, hence it contains a finitely generated essential submodule $G$. Now, Lemma~\ref{sing} suggests that $M/G$ is singular, because $$G\oplus S \trianglelefteq H\oplus S \trianglelefteq M.$$

(2) $\Rightarrow$ (3). This implication is straightforward.

(3) $\Rightarrow$ (1). First, observe that every submodule of $M$ also satisfies condition (3). Let $N$ be a submodule of $M$, and $P$ a maximal invertible ideal, then $\Sing(N)=S\cap N$, whence the $P$-primary component of $\Sing(N)$ is a submodule of the $P$-primary component of $S$. As $R$ is Noetherian, $S(P)$, and so $\Sing(N)(P)$, is Noetherian as well. The existence of finitely generated submodule $H$ of $N$ such that $N/H$ is singular follows now from Lemma~\ref{hereditarily}.

Resuming, we have just shown that it suffices to prove, without loss of generality, that $M$ is Bassian by application of \cite[Theorem 3]{T}. In fact, we have already noted that each primary component of $S$ is Noetherian. Since $M/G$ is singular, we get $G+S\trianglelefteq M$, and so $(G+S)/S\trianglelefteq M/S$ in virtue of Lemma~\ref{sing}. Furthermore, as $(G+S)/S$ is finitely generated, the quotient $M/S$ has finite dimension. We thus have verified the hypothesis of \cite[Theorem 3]{T}, implying that $M$ is Bassian, and we are set.
\end{proof}

As two subsequent consequences, we deduce:

\begin{cor}\label{cor-Dedekind} Let $M$ be a module over a non-primitive Dedekind prime ring.
\begin{enumerate}
\item[(1)] If $M$ is super Bassian, then $M$ is hereditarily Bassian.
\item[(2)] If $M$ is singular, the following conditions are equivalent:
\begin{enumerate}
\item[(a)] $M$ is hereditarily Bassian.
\item[(b)] $M$ is super Bassian.
\item[(c)] $M$ is Bassian.
\end{enumerate}
 \item[(3)] A non-singular module $M$ is hereditarily Bassian if, and only if, $M$ is of finite Goldie dimension.
\end{enumerate}
\end{cor}

\begin{proof} (1) Let $M$ be super Bassian and set $S:=\Sing(M)$. Then, Theorem~\ref{super Bassian} insures that there exists a finitely generated submodule $H$ of $M$ such that $M/H$ is singular, $H\cap S=(0)$, and each primary component of $M/H$ is Noetherian. Since $S$ is embeddable into $M/H$, each primary component of $S$ is Noetherian as well, and hence $M$ is hereditarily Bassian in view of Theorem~\ref{hereditarily Bassian}.

(2) It follows from combination of Corollary~\ref{sB}(1) and Theorem~\ref{hereditarily Bassian}.

(3) It is obvious.
\end{proof}

The following assertion reformulates Theorem~\ref{hereditarily Bassian} in the commutative case.

\begin{cor} A module $M$ over a Dedekind domain is hereditarily Bassian if, and only if, all primary components of $M$ are finitely generated and there exists a finitely generated torsion-free submodule $F$ of $M$ such that $M/F$ is torsion.
\end{cor}

In case of Abelian groups, one can say that {\it a (reduced) group is Bassian if, and only if, it is hereditarily Bassian} (compare with \cite{DG} and \cite{K}).

\medskip

We now finish this section with two examples which illustrates that the notions of super Bassian and hereditarily Bassian modules are independent one another, thus motivating a serious further research examination in order to obtain their complete structural characterization.

\begin{exm} The Abelian group $\Q$ of rank 1 is hereditarily Bassian looking at Corollary~\ref{cor-Dedekind}(3), and since it is well-known that $\Q/\Z\cong \bigoplus_p\Z_{p^\infty}$, it follows at once that it need not be super Bassian adapting Corollary~\ref{sBcom}.
\end{exm}

\begin{exm} Let $K$ be field, $V$ a vector space over $K$ of infinite dimension, and denote by $R$ a trivial extension of $K$ by $V$, that is, $R/J(R)\cong K$, $J(R)\cong V$ and $J(R)^2=(0)$. Then, $R$ is a local perfect ring such that the module $R_R$ is super Bassian (and so Hopfian), but $R_R$ is not hereditarily Bassian as $J(R)$ is not Bassian.
\end{exm}

In closing, we pose the following two challenging questions that could be of some interest and importance for a widely audience of researchers (see cf. \cite{K}).

\begin{prob} Describe the structure of hereditarily and super Bassian modules over larger classes of (possibly non-commutative) rings and domains.
\end{prob}

In regard to \cite{B} and \cite{RS}, the following seems to make sense.

\begin{prob} Describe the structure of hereditary and super Bassian rings.
\end{prob}

\medskip
\medskip

\section*{Data availability}
No data was used for the research described in the article.

\medskip
\medskip

\section*{Declarations}
The authors declare no any conflict of interests while writing and preparing this manuscript.


\vskip2.0pc

\begin{thebibliography}{99}

\bibitem{B} N. Blacher, {\it Rings whose subrings are all Noetherian or Artinian}, J. Algebra {\bf 693} (2026), 362--371.

\bibitem{CDG} A. R. Cheklov, P. V. Danchev and B. Goldsmith, {\it On the Bassian property for Abelian groups}, Arch. Math. (Basel) {\bf 117}(6) (2021), 593--600.

\bibitem{DG} P. Danchev and B. Goldsmith, {\it Super Bassian and nearly generalized Bassian Abelian groups}, Internat. J. Algebra Comput. {\bf 34}(5) (2024), 639--653.

\bibitem{DKTZ} S. Das, M. T. Ko\c{s}an, O. Ta\c{s}demir and J. \v{Z}emli\v{c}ka, {\it On Bassian and generalized Bassian modules}, J. Algebra Appl. {\bf 25} (2026).

\bibitem{K} P. W. Keef, {\it Sub-Bassian properties on Abelian groups}, Commun. Algebra {\bf 53}(4) (2025), 1723--1738.

\bibitem{RS} L. Rowen and L. Small, {\it Hopfian and Bassian algebras}, arXiv:1711.06483v2.

\bibitem{T} A. Tuganbaev, {\it On Bassian Modules}, Mathematics MDPI {\bf 14} (2026).
\\ https://doi.org/10.3390/math14010108

\end{thebibliography}
\end{document}